\documentclass[12pt]{amsart}
\usepackage[dvipdfmx]{graphicx}
\title[Nielsen equivalence and multisections of 4-manifolds]
{Nielsen equivalence and multisections of 4-manifolds}

\author{Tsukasa Isoshima}
\address{Department of Mathematics, Tokyo Institute of Technology, 2-12-1 Ookayama, Meguro-ku, Tokyo, 152-8551, Japan}
\email{isoshima.t.aa@m.titech.ac.jp}

\author{Masaki Ogawa}
\address{Mathematical science center for co-creative society, Tohoku University, Aoba-6-3 Aramaki, Aoba Ward, Sendai, Miyagi 980-0845}
\email{masaki.ogawa.b7@tohoku.ac.jp}

\usepackage[utf8]{inputenc}
\usepackage{amsmath,amssymb,amsthm,amsfonts}
\usepackage{epsfig}
\usepackage{txfonts}
\usepackage{graphicx}
\usepackage{ulem}
\usepackage{color}
\usepackage[all]{xy}
\usepackage{fancybox}
\usepackage{lscape}
\usepackage{comment}
\usepackage{here}
\usepackage{setspace}
\usepackage{bm}

\usepackage[hang,small,bf]{caption}
\usepackage[subrefformat=parens]{subcaption}
\theoremstyle{plain}
\newtheorem{theorem}{Theorem}[section]
\newtheorem{lemma}[theorem]{Lemma}
\newtheorem{corollary}[theorem]{Corollary}
\newtheorem{proposition}[theorem]{Proposition}

\theoremstyle{definition}
\newtheorem{definition}[theorem]{Definition}

\newtheorem{remark}[theorem]{Remark}
\newtheorem{example}[theorem]{Example}

\theoremstyle{definition}

\definecolor{myred}{rgb}{.8,.0,.0}
\definecolor{mygreen}{rgb}{.0,.6,.0}
\definecolor{mygray}{gray}{0.7}

\makeatletter 
 \@mparswitchfalse 
\makeatother

\newcounter{mystepcount}

\begin{document}

\maketitle

\begin{abstract}
	Islambouli showed that there exist infinitely many 4-manifolds admitting non-isotopic trisections using a Nielsen equivalence, which can be used to construct non-isotopic Heegaard splittings. In this paper, we show that there exist infinitely many 4-manifolds admitting non-isotopic bisections in the same way. Moreover, we show that there exist infinitely many 4-manifolds admitting non-isotopic multisections by using the non-isotopic bisections in two ways.
\end{abstract}

\section{Introduction}
A trisection (resp. relative trisection) is roughly a decomposition of a closed 4-manifold (resp. a 4-manifold with boundary) into three 4-dimensional 1-handlebodies. These decompositions were introduced by Gay and Kirby \cite{GK}. They showed that any closed 4-manifold admits a trisection, and any 4-manifold with boundary admits a relative trisection. After that, Islambouli and Naylor \cite{IN} introduced a multisection (or $n$-section) of a closed 4-manifold which decomposes the 4-manifold into $n$ 1-handlebodies for any integer $n \ge 3$. When $n=2$ and the 4-manifold has a boundary, this decomposition is called a bisection. It is known that any smooth, compact, connected 2-handlebody with connected boundary admits a bisection \cite{IN}. In this paper, we will focus on bisections and 4-sections in particular.

A trisection can be regarded as a 4-dimensional analogy of a Heegaard splitting which is a decomposition of a closed 3-manifold into two 3-dimensional 1-handledodies. The classification of Heegaard splittings is one of main problems in studying Heegaard splittings, and there are several works on the classification. For example, it is well known by Waldhausen that two Heegaard splittings of $S^3$ with the same genus are isotopic \cite{W}. As a 4-dimensional analogy of this result, it is conjectured by Meier, Schirmer, and Zupan that two trisections of $S^4$ with the same type are isotopic \cite{MSZ}, which is called 4-dimensional Waldhausen's conjecture. Note that two trisections of the same 4-manifold are isotopic after some stabilizations \cite{GK}. In the present, there are some examples of complicated trisections isotopic to the standard trisection of $S^4$ (e.g. \cite{IS}). For 4-manifolds other than $S^4$, it is shown by Islambouli \cite{I} that there exist infinitely many 4-manifolds with non-isotopic trisections. He uses a Nielsen equivalence to prove the result which was originally used to show the existence of non-isotopic Heegaard splittings. In this paper, we show a bisectional and multisectional analogy of Islambouli's result in the same way. 

Our first result is the following:

\setcounter{section}{4}
\setcounter{theorem}{5}
\begin{theorem}
	For every $n\geq 2$, there exist 4-manifolds that admit non-isotopic $(2n, n)$-bisections of minimal genus.
\end{theorem}

In the proof of this theorem, we construct non-isotopic bisections of certain 4-manifolds from non-isotopic Heegaard splittings. The second result can be obtained by considering other bisections constructed obviously from trisections.

\setcounter{section}{5}
\setcounter{theorem}{2}
\begin{corollary}
	For every $n\geq 2$, there exist 4-manifolds that admit non-isotopic $(3n, n)$-bisections.
\end{corollary}

The third result can be obtained by considering 4-sections constructed by doubling the bisections in Theorem \ref{thm:bisec} as non-isotopic multisections produced immediately from the non-isotopic bisections.

\setcounter{section}{6}
\setcounter{theorem}{2}
\begin{theorem}
	For every $n\geq 2$, there exist 4-manifolds which admit non-isotopic $(2n, n)$ 4-sections of minimal genus.
\end{theorem}

As a generalization of Theorem \ref{thm:4-sec}, we lastly construct non-isotopic multisections of 4-manifolds with any number of sectors in two ways.

The first way is that we construct an $n$-section of a 4-manifold with boundary by modifying the bisection. Then, we have the following:

\setcounter{section}{7}
\setcounter{theorem}{2}
\begin{theorem}
	For every $n\geq 2$ and $m\geq 4$, there exist 4-manifolds that admit non-isotopic genus $2n$ m-sections.
\end{theorem}

The second way is that we construct a $2n$-section of a 4-manifold with boundary by gluing $n$ bisections in Theorem \ref{thm:bisec} and glue the $2n$-section and the one bisection. Then, we have the following:

\setcounter{section}{8}
\setcounter{theorem}{5}
\begin{theorem}
For every $n \ge 2$ and $m \ge 2$, there exist 4-manifolds that admit non-isotopic $(2n,n)$ $2m$-sections of the minimal genus.
\end{theorem}

We can regard the $2n$-section as a $2n-1$-section by thinking of two adjacent sectors as one sector. Then, we have the following:

\setcounter{section}{8}
\setcounter{theorem}{6}
\begin{theorem}
For every $n \ge 2$ and $m \ge 2$, there exist 4-manifolds that admit non-isotopic $(2n,n)$ $2m+1$-sections of the minimal genus.
\end{theorem}

This paper is organized as follows: 
In Section \ref{sec:prelimi}, we review the definition of trisections, multisections and their diagrams, and their equivalence classes. In Section \ref{sec:nielsen}, we recall the Nielsen equivalence and a basic property of bisections and multisections with respect to the Nielsen equivalence. We show the above theorems in Sections \ref{sec:bisect_from_HS}, \ref{sec:bisec_from_trisec}, \ref{sec:4-sec}, \ref{sec:multi0} and \ref{sec:multi}.

\setcounter{section}{1}
\section{preliminalies}\label{sec:prelimi}
In this section, we set up the notions and objects we use in this paper.
Throughout this paper, we suppose any 4-manifold is smooth, compact, connected, and oriented.
 \begin{definition}
Let $g,k_1,k_2$ and $k_3$ be non-negative integers with $\max\{k_1,k_2,k_3\}\leq g$. 
A $(g; k_1, k_2, k_3)$-{\it trisection} of a closed 4-manifold $X$ 
is a decomposition $X=X_1\cup X_2\cup X_3$ such that for $i,j\in\{1,2,3\}$, 
\begin{itemize}
 \item
$X_i\cong \natural^{k_i}(S^1\times B^3)$, 
 \item
$H_{ij}=X_i\cap X_j\cong \natural^g (S^1\times B^2)$ if $i\neq j$, and 
 \item
$\Sigma=X_1\cap X_2\cap X_3 \cong \#^g (S^1\times S^1)$. 
\end{itemize}
\end{definition}

The classification problem of trisections is as follows.

\begin{definition}
	Two trisections $X_1\cup X_2\cup X_3$ and $X_1'\cup X_2'\cup X_3'$ of $X$ are {\it isotopic} if there is an amibient isotopy $\{ \varphi_t \}$ for $t\in [0, 1]$ such that $\varphi_1(X_i)=X_i'$ for $i=1, 2, 3$.
	Two trisections above are {\it diffeomorphic} if there is a self-diffeomorphism $\psi : X \to X$ of $X$ such that $\psi(X_i)=X_i'$ for $i=1, 2, 3$.
\end{definition}

In this setting, there are several results on the classification of trisections. First, it is conjectured, called the 4-dimensional Waldhausen conjecture, that each trisection of the 4-sphere is isotopic to the stabilization of the genus 0 trisection. For this conjecture, the authors considered the trisections obtained by Gluck twisting along certain spun knots in \cite{IO} and the first author found infinitely many trisections for Gluck twisting that are diffeomorphic to the stabilization of the genus $0$ trisection of $S^4$ \cite{I}. 

On the other hand, Islambouli showed that there are infinitely many 4-manifolds that admit non-isotopic trisections in \cite{I}.
He constructed non-isotopic trisections by considering the spins of non-isotopic Heegaard splittings and showed they are non-isotopic by using the Nielsen equivalence of generators of the fundamental group. We recall the Nielsen equivalence in Section 2.

In this paper, we consider the following decomposition of a 4-manifold which is called a multisection. 
This is defined by Islambouli and Naylor in \cite{IN}. 
 \begin{definition}
For $n\in \mathbb{Z}_{\geq 3}$, let $g,k_1,k_2, \ldots,k_{n-1}$ and  $k_n$ be non-negative integers with $k_i \leq g$ for $i\in \{1,\dots n\}$. 
A $(g; k_1, k_2, \ldots, k_n)$-{\it n-section} of a closed 4-manifold $X$ 
is a decomposition $X=X_1\cup X_2\cup \cdots \cup X_n$ such that for $i,j\in\{1,2,\ldots, n\}$, 
\begin{itemize}
 \item
$X_i\cong \natural^{k_i}(S^1\times B^3)$, 
 \item
$H_{ij}=X_i\cap X_j\cong \natural^g (S^1\times B^2)$ if $|i - j|=1$ otherwise $X_i\cap X_j\cong \Sigma_g$
 \item
$\partial X_i=H_{i-1, i}\cup H_{i, i+1}$ gives a genus $g$ Heegaard splitting of $\partial X_i$. 
\end{itemize}

We call $\Sigma_g$ the {\it central surface}, $X_i$ the \textit{sector} for each $i$ and $g$ the {\it genus} of the multisection.
The union $H_{ij}\cup H_{kl}$ is a Heegaard splitting of genus $g$ for any $| i - j |,  |k - l | =1$. If $H_{ij}\cup H_{kl}$ is not the boundary $\partial X_m$ for any $m\in \{1, \ldots , n\}$, it is called the {\it cross-section}
\end{definition}

The decomposition of an $n$-section is described by an $n$-section diagram.

\begin{definition}
A {\it $(g; k_1, k_2, \ldots, k_n)$-$n$-section diagram} is an ordered $(n+1)$-tuple $(\Sigma_g, c_1,c_2, \ldots, c_n)$ such that $(\Sigma_g, c_i, c_{i+1})$ is a Heegaard diagram of $\#_{k_i} S^1 \times S^2$ for $i=1,2,\ldots, n$, where $c_{n+1}=c_1$.
\end{definition}


The equivalence classes of multisections are defined as follows as well as them of trisections.
\begin{definition}
		Two multisections $X_1\cup X_2\cup \cdots \cup X_n$ and $X_1'\cup X_2'\cup \cdots \cup X_n'$ of $X$ are {\it isotopic} if there is an amibient isotopy \{$\varphi_t$\} for $t\in [0, 1]$ such that $\varphi_1(X_i)=X_i'$ for $i=1, 2, \ldots , n$.
	Two multisections above are {\it diffeomorphic} if there is a self diffeomorphism $\psi : X \to X$ of $X$ such that $\psi(X_i)=X_i'$ for $i=1, 2, \ldots , n$.
\end{definition}

If two multisections have different cross-sections up to diffeomorphism, they are not diffeomorphic to each other.

\begin{definition}\label{def:bisec}
	Let $k_1$, $k_2$, and $g$ be non-negative integers such that $k_i\leq g$ for $i=1, 2$ and $X$ a 4-manifold with connected boundary.
	A $(g; k_1, k_2)$-{\it bisection} of $X$ is a decomposition $X=X_1\cup X_2$ such that
	\begin{itemize}
		\item $X_i\cong \natural^{k_i} S^1\times B^3$,
		\item $\partial X_1=H_1\cup H_2$ and $\partial X_2=H_2\cup H_3$ are genus $g$ Heegaard splittings of $\#^{k_1}S^1\times S^2$ and $\#^{k_2}S^1\times S^2$ respectively.
		\item $\partial X = H_1\cup H_3$ is a genus $g$ Heegaard splitting of $\partial X$.
	\end{itemize}
	We call $g$ the \textit{genus} of the bisection $\mathcal{B}$ and it is denoted by $bg(\mathcal{B})$, and we also define the \textit{bisection genus} of $X$ by 
	\[
		bg(X)= min \{bg(\mathcal{B})\mid \mbox{$\mathcal{B}$ is a bisection of $X$}\}.
	\]
\end{definition}


We give the bisection genus of a certain 4-manifold in Section 4. 

Islambouli and Naylor showed that any 2-handlebody, that is, a 4-manifold that admits a handle decomposition consisting of 0-, 1-, and 2-handles, admits a bisection in \cite{IN}. 
Note that a bisected closed 4-manifold is diffeomorphic to $\#_{k}S^1 \times S^3$ for some $k \ge 0$ by Laudenbach-Po\'{e}naru's theorem \cite{LP}. Thus, it is essential that we consider a 4-manifold with connected boundary in Definition \ref{def:bisec}.

The decomposition of a bisection can be described by a bisection diagram.

\begin{definition}
A $(g; k_1, k_2)$-{\it bisection diagram} is an orderd 4-tuple $(\Sigma_g, \alpha, \beta, \gamma)$ such that $(\Sigma_g, \alpha, \beta)$ is a Heegaard diagram of $\#_{k_1}S^1 \times S^2$ and $(\Sigma_g, \beta, \gamma)$ is a Heegaard diagram of $\#_{k_2}S^1 \times S^2$.
\end{definition}

Note that $(\Sigma_g, \gamma, \alpha)$ is a Heegaard diagram of the boundary of the 4-manifold corresponding to the bisection diagram. One can depict multisection (resp. bisection) diagrams from multisections (resp. bisections). See \cite{IN} for details.

To consider the classification problems of bisections, we introduce a notion of equivalence classes of bisections.

\begin{definition}
	Two bisections $X_1\cup X_2$ and $X_1'\cup X_2'$ of $X$ are {\it isotopic} if there is an amibient isotopy $\{ \varphi_t \}$ for $t\in [0, 1]$ such that $\varphi_1(X_i)=X_i'$ for $i=1, 2$.
	Two bisections above are {\it diffeomorphic} if there is a self diffeomorphism $\psi : X \to X$ of $X$ such that $\psi(X_i)=X_i'$ for $i=1, 2$.
\end{definition}

For the above notion, some examples of non-diffeomorphic 4-sections have been found in \cite{IN}. We recall this briefly. 

First of all, we recall Mazur manifolds introduced in \cite{M}. The construction of Mazur manifolds is as follows:
Let $W_1$ be a 1-handlebody diffeomorphic to $S^1\times B^3$ and $W_2$ a 2-handle attached to $W_1$.
Suppose that the attaching circle of $W_2$ represents the generator of $H_1(\partial W_1)$.
Then, the union $W_1\cup W_2$ is called a Mazur manifold. 

Mazur showed that manifolds obtained by the above construction are contractible and their boundaries are homology 3-spheres.
A significant feature of this manifold is that the double of this manifold is diffeomorphic to $S^4$.

\begin{remark}[{\cite[Section 6.3]{IN}}]
It is shown that there is an exotic pair of Mazur manifolds \cite{HMP}. 
Since the double of a Mazur manifold is $S^4$, we can construct a 4-section of $S^4$ by doubling a bisection of a Mazur manifold. 
The 4-sections obtained by this construction for these exotic pairs are not diffeomorphic. 
Thus, 4-dimensional Waldhausen's conjecture is not correct for 4-sections of $S^4$. 
\end{remark}

This construction of non-diffeomorphic 4-sections relies on an exotic pair of Mazur manifolds.
In Section \ref{sec:4-sec}, we construct non-isotopic 4-sections of 4-manifolds which are not diffeomorphic to $S^4$ in a different way. 

\section{Nielsen equivalence and classification of bisections}\label{sec:nielsen}

In this section, we recall a notion of Nielsen equivalence. The Nielsen equivalence is an equivalence relation among generating sets of finitely generated groups. We use the Nielsen equivalence to construct mutually non-isotopic bisections.

\begin{definition}
	Let $X=(x_1, \ldots , x_n)$ be a basis of the free group $F_n$ of rank $n$ and $G$ a finitely generated group with two generating sets $(a_1, \ldots , a_n)$ and $(b_1, \ldots, b_n)$.
	We say that $(a_1, \ldots , a_n)$ and $(b_1, \ldots, b_n)$ are {\it Nielsen equivalent} if there exist another basis $(y_1, \ldots , y_n)$ of $F_n$ and a homomorphism $\varphi: F_n \to G$ such that $\varphi(x_i)=a_i$ and $\varphi(y_i)=b_i$.
\end{definition}

The Nielsen equivalence can be reformulated by a result of Nielsen \cite{N} as follows:

\begin{definition}[reformulated definition of Nielsen equivalence]
	Let $G$ be a finitely generated group with generating sets $(a_1, \ldots, a_n)$, $(b_1, \ldots, b_n)$ and $w_i$ a word of $b_i$ represented by $(a_1, \ldots, a_n)$.
	We say $(a_1, \ldots, a_n)$ and $(b_1, \ldots, b_n)$ are {\it Nielsen equivalent} if $(w_1, \ldots , w_n)$ can be obtained from $(a_1, \ldots, a_n)$ by applying the following operations.
	\begin{enumerate}
		\item Swap $a_1$ and $a_2$.
		\item Permute $(a_1, a_2, \ldots, a_n)$ to $(a_2, a_3 \ldots, a_n, a_1)$.
		\item Replace $a_1$ with $a_1^{-1}$.
		\item Replace $a_1$ with $a_1a_2$.
	\end{enumerate}
\end{definition}
	
	A spine of a 1-handlebody is a graph which is a deformation retract of the 1-handlebody. 
	In this paper, following \cite{I}, we suppose any spine of a 1-handlebody has only one vertex.
	
	Let $H_g$ be a 4-dimensional 1-handlebody embedded in $X$ and $S$ a spine of $H_g$ with a vertex $v$ and edges $s_1, \ldots, s_g$.
	Suppose that there is an epimorphism $i_\ast : \pi(H_g) \to \pi(X)$.
	Then, we denote by $\mathcal{N}(H_g)$ the Nielsen class of $g$ generators of $\pi_1(X, x_0)$ given by
	\[
		(i_\ast(p\cdot s_1\cdot p^{-1}), \ldots, i_\ast(p\cdot s_g\cdot p^{-1})),
	\]
	where $p$ is a path from $x_0$ to $v$.
	
	Let $f$ be a self-diffeomorphism of $X$ which fixes a base point $x_0$. 
	Then we denote by $f(\mathcal{N}(H_g))$ the Nielsen class which is obtained by applying $f_\ast$, where $f_\ast$ is a self-isomorphism of $\pi_1(X)$ induced from $f$.


The following lemma shows that the Nielsen equivalence can be used to detect different bisections.
This is a similar result to the one for trisections \cite[Proposition 4.5]{I}.

\begin{lemma}\label{lem:funddam1}
	Let $X_1\cup X_2$ and $X_1'\cup X_2'$ be two $(g, k)$ bisections of $X$. If these bisections are isotopic, then $\mathcal{N}(X_i)=\mathcal{N}(X_i')$. If the bisections are diffeomorphic by some diffeomorphism $f$, then $f(\mathcal{N}(X_i))=\mathcal{N}(X_i')$ for $i=1, 2$.
\end{lemma}

\begin{proof}
	If $X_1\cup X_2$ and $X_1'\cup X_2'$ are isotopic, then spines of $X_i$ and $X_i'$ are isotopic to each other in $X$ for $i=1, 2$. By Lemma 3.2 and 3.1 in \cite{I}, the generators of $\pi_1(X)$ coming from spines of $X_i$ and $X_i'$ are Nielsen equivalent. If they are diffeomorphic, there is a diffeomorphism $f$ which sends $X_i$ to $X_i'$. Then we send the generator of $\pi_1(X)$ coming from the spine of $X_i$ by the self-isomorphism $f_\ast$ of $\pi_1(X)$ induced from $f$. Then this generator is Nielsen equivalent to the generator coming from $X_i'$ by Lemma 3.2 and 3.1 in \cite{I}.
\end{proof}

For multisections, we also obtain the following:
\begin{lemma}\label{lem:fund2}
	Let $X_1\cup X_2\cup \cdots \cup X_n$ and $X_1'\cup X_2'\cup \cdots \cup X_n'$ be two $(g, k)$ multisections of $X$. If these multisections are isotopic, then $\mathcal{N}(X_i)=\mathcal{N}(X_i')$. If the multisections are diffeomorphic by some diffeomorphism $f$, then $f(\mathcal{N}(X_i))=\mathcal{N}(X_i')$ for $i=1, 2$.
\end{lemma}

\begin{proof}
	If $X_1\cup X_2\cup \cdots \cup X_n$ and $X_1'\cup X_2'\cup \cdots \cup X_n'$ are isotopic, then spines of $X_i$ and $X_i'$ are isotopic to each other in $X$ for $i=1, 2, \ldots, n$. By Lemma 3.2 and 3.1 in \cite{I}, the generators of $\pi_1(X)$ coming from spines of $X_i$ and $X_i'$ are Nielsen equivalent. If they are diffeomorphic, there is a diffeomorphism $f$ which sends $X_i$ to $X_i'$. Then we send the generator of $\pi_1(X)$ coming from the spine of $X_i$ by the self-isomorphism $f_\ast$ of $\pi_1(X)$ induced from $f$. Then this generator is Nielsen equivalent to the generator coming from $X_i'$ by Lemma 3.2 and 3.1 in \cite{I}.
\end{proof}

\section{a bisection obtained from a Heegaard splitting}\label{sec:bisect_from_HS}
In this section, we construct a bisection from a Heegaard splitting and give infinite  families of 4-manifolds with non-isotopic bisections.
Firstly, we construct a bisection from a Heegaard splitting.

Let $M$ be an oriented closed 3-manifold and $V_1\cup V_2$ a genus $g$ Heegaard splitting of $M$.
We take a point $p$ on the Heegaard surface $\Sigma_g$ of $V_1\cup V_2$.
Let $N(p)$ be a regular neighborhood of $p$ in $M$ and consider the 4-manifold $X=(M\setminus N(p))\times I$, where $I$ is the interval $[0, 1]$.

We define a decomposition of $X=X_1 \cup X_2$ as follows:
\begin{itemize}
	\item $X_1=(V_1\setminus N(p))\times I$, and 
	\item $X_2=(V_2\setminus N(p))\times I$.
\end{itemize}
Figure \ref{bisection1} shows a schematic picture of this decomposition.
	\begin{figure}[htbp]
		\centering
		\includegraphics[scale=0.6]{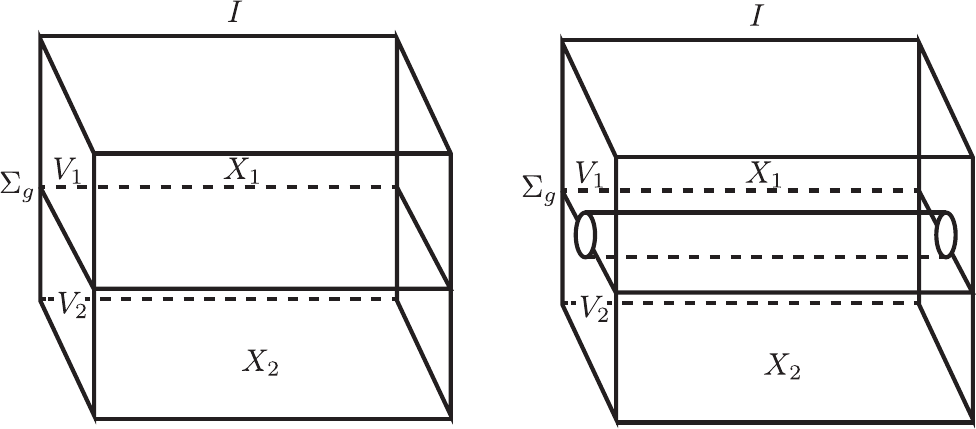}
		\caption{This figure is a schematic picture of the bisection obtained by taking a product of a punctured Heegaard splitting and the interval.}
		\label{bisection1}
	\end{figure}
\begin{proposition}\label{prop:construction}
	The decomposition $X=X_1\cup X_2$ is a genus $2g$ bisection of $X$.
\end{proposition}
\begin{proof}
	To show the statement, we check the following:
	\begin{itemize}
		\item Each of $X_i$ is a 1-handlebody for $i=1, 2$.
		\item Each of $\partial X_i = (X_1\cap X_2)\cup (X_i\setminus X_1\cap X_2)$ is a Heegaard splitting of genus $2g$ for $i=1, 2$.
		\item The decomposition $\partial X = (X_1\setminus X_1\cap X_2)\cup (X_2\setminus X_1\cap X_2)$ is a Heegaard splitting of genus $2g$.
	\end{itemize}
	
	Since a product of a 3-dimensional handlebody and the interval is a 4-dimensional handlebody, $X_i$ is a 4-dimensional handlebody. Furthermore, since the genus of the 3-dimensional handlebody $V_i\setminus N(p)$ is $g$, $X_i$ is diffeomorphic to $\natural^{g} S^1\times B^3$. 
	
	By the construction, $X_1\cap X_2 = (\Sigma_g \setminus N(p)) \times I$, and $(\Sigma_g \setminus N(p)) \times I$ is a genus $2g$ 3-dimensional handlebody since $\Sigma_g \setminus N(p)$ is a union of a $0$-handle and $2g$ $1$-handles. 
	The intersection $X_1\cap X_2$ is denoted by $H_2$.
	
	By the construction, the following holds:
	\[
		\overline{X_i \setminus H_2} = ((V_i \setminus N(p))\times \{0\}) \cup ((V_i\cap N(p))\times I) \cup ((V_i \setminus N(p))\times \{1\}).
	\]
	Here, $((V_i\cap N(p))\times I)$ is a 3-dimensional 1-handle connecting $((V_i \setminus N(p))\times \{0\})$ and $((V_i \setminus N(p))\times \{1\})$. Hence, this is a genus $2g$ 3-dimensional handlebody. The sets $\overline{X_1 \setminus H_2}$ and $\overline{X_2 \setminus H_2}$ are denoted by $H_1$ and $H_3$, respectively.
	To prove that $\partial X_i = (X_1\cap X_2)\cup (X_i\setminus X_1\cap X_2)$ is a Heegaard splitting, it suffices to show that $H_i\cap H_2$ is a genus $2g$ closed surface for $i=1, 3$. This follows from
	\[
		H_2\cap H_i = ((\Sigma_g \setminus N(p))\times \{0\}) \cup ((\Sigma_g \setminus N(p))\times I) \cup ((\Sigma_g \setminus N(p))\times \{1\}).
	\]
	This is a genus $2g$ closed surface. Hence, $H_1\cup H_2$ and $H_3\cup H_2$ are genus $2g$ Heegaard splittings of $\partial X_1$ and $\partial X_2$, respectively. 
	
	Finally, we check that $H_1\cup H_3$ is Heegaard splitting of $\partial X$.
	Since $X=X_1\cup X_2$, $\partial X= (\partial X_1\cup \partial X_2)\setminus (\partial X_1\cap \partial X_2)$.
	Hence $\partial X = H_1\cup H_3$.
	The intersection $H_1\cap H_3$ is the same as the intersection $H_2\cap H_i$.
\end{proof}

We can algorithmically depict the bisection diagram with respect to the above bisection as follows:

\begin{enumerate}
	\item Since the boundary of the bisection is the connected sum of two copies of the given Heegaard splitting $V_1\cup V_2$, the boundary of meridian disk systems of $H_1$ (resp. $H_3$) corresponds to the boundary of meridian disk system of $V_1\natural (-V_1)$ (resp. $V_2\natural (-V_2)$). That is, the boundary of meridian disk systems of $H_1$ and $H_3$ constitutes a Heegaard diagram which is a connected sum of the Heegaard diagram of $M$ and $-M$.
	\item Since $H_2$ is the product of a once punctured Heegaard surface of $V_1\cup V_2$ and the interval $I$, we can take a meridian disk system of $H_2$ as the product of the cocore of 1-handles of $\Sigma\setminus N(p)$ and $I$.
	\item Lastly, we depict the boundary of each meridian disk of $H_2$ on the central surface.
\end{enumerate}

\begin{example}
Figure \ref{bidiag1} represents the bisection diagram with respect to the above bisection of $L(2, 1)^\circ \times I$. More generally, let $\alpha$ and $\beta$ be simple closed curves in Figure \ref{bidiag1}. Also, let $\gamma_1$ be a simple closed curve that only intersects with the left $\alpha$ curve (we call it $\alpha_1$) $p$ times and the longitude curve dual to $\alpha_1$ $q$ times, and $\gamma_2$ the mirror image of $\gamma_1$. Then, $(\Sigma, \alpha, \beta, \gamma)$ is the bisection diagram with respect to the above bisection of $L(p, q)^\circ \times I$.
\end{example}

\begin{figure}[h]
		\centering
		\includegraphics[scale=1]{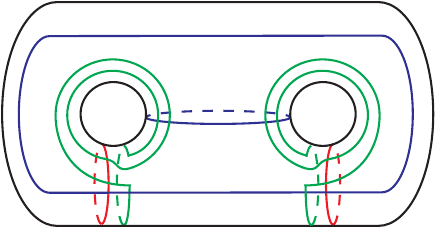}
		\caption{This figure is the bisection diagram with respect to the above bisection of $L(2, 1)^\circ \times I$.}
		\label{bidiag1}
	\end{figure}

As a consequence of Proposition \ref{prop:construction}, we obtain bisections with minimal genus.

\begin{corollary}\label{cor_1}
	Let $M$ be a closed $3$-manifold.
	Then, the bisection genus of $M^\circ \times I$ is $2g(M)$, where $g(M)$ is the Heegaard genus of $M$.
	Furthermore,  for any $n\geq 2$, there are infinitely many compact smooth 4-manifolds that each bisection genus  of them is $2n$.
\end{corollary}
	\begin{proof}
		The boundary of $M^\circ \times I$ is diffeomorphic to $M\# -M$.
		It is well-known that the Heegaard genus is additive under connected sum.
		Hence,  $g(M\# -M)=2g(M)$.
		Since the bisection genus of a 4-manifold is bounded by the Heegaard genus of its boundary from below, we obtain
		\[
			2g(M)\leq bg(M^\circ \times I).
		\]
		By the construction above, we can construct the genus $2g(M)$ bisection of $M^\circ \times I$ from a Heegaard splitting of $M$ with genus $g(M)$.
		
		By Proposition \ref{prop:construction}, we can construct genus $2n$ bisections of $\#^nL(p, q)^\circ \times I$ from a genus $n$ Heegaard splitting of $\#^n L(p, q)$ for any $p, q$.
		Since the Heegaard genus of $\#^n L(p, q)$ is $n$,  there are infinitely many compact smooth 4-manifolds that each bisection genus of them is $2n$ for $n\geq 2$.
	\end{proof}

	The bisection obtained from a Heegaard splitting $V_1\cup _{\Sigma} V_2$ is denoted by $B(\Sigma)$.
	Let $M$ be a 3-manifold with Heegaard splitting $V_1\cup V_2$.
	Since $\pi_1(M)$ is isomorphic to $\pi_1(M^\circ)$, $\pi_1(M^\circ\times I)$ is isomorphic to $\pi_1(M)$.
	Furthermore, the following lemma is used to construct non-isotopic bisections.
	Let $M^\circ$ be a 3-manifold obtained by removing an open ball from $M$.
	
\begin{lemma}\label{lem:spine}
	Let $V_1\cup_{\Sigma} V_2$ be a genus $g$ Heegaard splitting of $M$, $S_1$ a spine of $V_1$ and $\langle s_1, \dots, s_g\mid R\rangle$ a presentation of $\pi_1(M)$, where $s_i$ is a homotopy class of  loops in $S_1$.
	Then, $S_1\times \{1/2\}$ induces the same presentation for $\pi_1(X)$ where $X=M^{\circ} \times I$.
\end{lemma}

\begin{proof}
	Let $i_\ast: \pi_1(M^\circ)\rightarrow\pi_1(M)$ be the isomorphism induced by the inclusion $i : M^\circ \hookrightarrow M$.
	Since $M^\circ \times \{ 1 \}$ is a deformation retract of $X=M^\circ \times I$, there is an isomorphism $d_\ast :\pi_1(M^\circ) \rightarrow \pi_1(X)$ induced by the deformation retract. This retraction sends $S_1\times \{1/2\}$ to $S_1\times \{1\}$. This implies that the image under $i\circ d$ of $S_1\times \{1/2\}$ induces generators of $\pi_1(X)$ and it induces the same presentation for $\pi_1(X)$.
\end{proof}

We can reduce the construction of non-isotopic bisections of $X$ to construct non-isotopic Heegaard splittings by the following.

\begin{proposition}\label{prop:detect}
	Let $V_1\cup_{\Sigma} V_2$ and $V_1'\cup_{\Sigma'} V_2'$ be Heegaard splittings of the same 3-manifold $M$ and $B(\Sigma)$ and $B(\Sigma')$ bisections obtained from $V_1\cup V_2$ and $V_1'\cup V_2'$ respectively.
	If $\mathcal{N}(V_1)\neq \mathcal{N}(V_1')$, then $B(\Sigma)$ and $B(\Sigma^{'})$ are mutually non-isotopic.
\end{proposition}

\begin{proof}
	Let $S_1$ and $S_1'$ be spines of $V_1$ and $V_1'$ respectively. Since $X_1$ and $X_1'$ are $(V_1\setminus N(p))\times I$ and $(V_1'\setminus N(p'))\times I$, $S_1$ and $S_1'$ are isotopic to $S_1\times \{1/2\}$ and $S_1'\times \{1/2\}$ respectively.
	 Let $d$ and $d'$ be deformation retracts of $X_1$ and $X_1'$ to $(V_1\setminus N(p))\times \{1/2\}$ and $(V_1'\setminus N(p'))\times \{1/2\}$ respectively.
	 Since $S_1$ and $S_1'$ are spines of $V_1$ and $V_1'$, $S_1\times \{1/2\}$ and $S_1'\times \{1/2\}$ are deformation retracts of $(V_1\setminus N(p))\times \{1/2\}$ and $(V_1'\setminus N(p'))\times \{1/2\}$.
	 Hence $S_1\times \{1/2\}$ and $S_1'\times \{1/2\}$ are spines of $X_1$ and $X_1'$ respectively.
	 
	 By Lemma \ref{lem:spine}, we may indentify $\pi_1(M^\circ)$ with $\pi_1(X)$ so that these spines for $X_1$ and $X_1'$ are the same as them of $V_1$ and $V_1'$. Therefore $\mathcal{N}(V_1)=\mathcal{N}(X_1)$ and $\mathcal{N}(V_1')=\mathcal{N}(X_1')$. Since $\mathcal{N}(V_1)\neq \mathcal{N}(V_1')$, we obtain $\mathcal{N}(X_1)\neq \mathcal{N}(X_1')$. By Lemma \ref{lem:funddam1}, $B(\Sigma)$ and $B(\Sigma^{'})$ are mutually non-isotopic.
\end{proof}

By Proposition \ref{prop:detect}, we can construct non-isotopic bisections from a pair of non-isotopic Heegaard splittings.
Following Section 7 of \cite{I}, we have the following: 
\begin{theorem}\label{thm:bisec}
	For every $n\geq 2$, there exist 4-manifolds that admit non-isotopic $(2n, n)$-bisections of minimal genus.
\end{theorem}

\begin{remark}
We say that a bisection $X_1 \cup X_2$ is \textit{flippable} if the bisection $X_2 \cup X_1$ is isotopic to $X_1 \cup X_2$. By Proposition \ref{prop:detect}, one can show that there exist 4-manifolds admitting non-flippable bisections since there exist 3-manifolds admitting Heegaard splittings $V_1 \cup V_2$ with $\mathcal{N}(V_1) \not = \mathcal{N}(V_2)$ (see Lemma 3.29 in \cite{J}).

\end{remark}

\section{a bisection obtained from a trisection}\label{sec:bisec_from_trisec}
In this section, we consider bisections obtained from trisections.
Let $X=X_1\cup X_2\cup X_3$ be a trisection of $X$. Then,  $X_i\cap X_j$ is a bisection of $\overline{X\setminus X_k}$ for $\{i, j, k\}=\{1, 2, 3\}$.

\begin{lemma}\label{tri-bi}
Let $X=X_1\cup X_2\cup X_3$ be a trisection and $\overline{X\setminus X_3}=X_1\cup X_2$ the bisection obtained from the trisection $X_1\cup X_2\cup  X_3$.
Then, $\pi_1(X)$ and $\pi_1(X\setminus X_3)$ are isomorphic, and furthermore, the generator given by a spine of $X_1$ induces the same presentation for $\pi_1(X)$ and $\pi_1(X\setminus X_3)$.
\end{lemma}

\begin{proof}
	By Lemma 13 in \cite{GK}, $X_1$ can be a union of a $0$-handle and $1$-handles for a handle decomposition of $X$. Hence a spine of $X_1$ gives a generator of $\pi_1(X)$. 
	Also, by Lemma 13 in \cite{GK}, $X_2$ contains all the 2-handles of the handle decomposition and $X_3$ contains the 3- and 4-handles of $X$. 
	Hence the union $X_1\cup X_2$ determines the fundamental group of $X$.
	Therefore, there is an isomorphism $\varphi: \pi_1(X)\to \pi_1(\overline{X\setminus X_3})$.
	Furthermore,  the generator of $\pi_1(X)$ given by the spine of $X_1$  induces the same presentation for $ \pi_1(\overline{X\setminus X_3})$.
\end{proof}

Let $B(\mathcal{T})$ be the above bisection obtained from a trisection $\mathcal{T}$, where $\mathcal{T}=X_1\cup X_2\cup X_3$.

\begin{proposition}\label{prop:detect2}
	Let $\mathcal{T}=X_1\cup X_2\cup X_3$ and $\mathcal{T}'=X_1'\cup X_2'\cup X_3'$ be trisections of the same 4-manifold $X$ and $B(\mathcal{T})$ and $B(\mathcal{T}')$ bisections obtained from $\mathcal{T}$ and $\mathcal{T}'$ respectively.
	If $\mathcal{N}(X_1)\neq \mathcal{N}(X_1')$, then $B(\mathcal{T})$ and $B(\mathcal{T}')$ are mutually non-isotopic.
\end{proposition}

\begin{proof}
	By Lemma \ref{tri-bi}, we may identify $\pi_1(X)$ with $\pi_1(X\setminus X_3)$ so that spines of the trisections for $X_1$ and $X_1'$ can be identified with them of the bisections.
         Since $\mathcal{N}(X_1)\neq \mathcal{N}(X_1')$ in $\pi_1(X)$,  we also have $\mathcal{N}(X_1)\neq \mathcal{N}(X_1')$ in $\pi_1(X\setminus X_3)$. By Lemma \ref{lem:funddam1}, $B(\mathcal{T})$ and $B(\mathcal{T}')$ are not isotopic to each other.
\end{proof}

By Proposition \ref{prop:detect2}, we obtain the following corollary by considering non-isotopic trisections introduced in \cite[Corollary 7.2]{I}.
\begin{corollary}
	For every $n\geq 2$, there exist 4-manifolds that admit non-isotopic $(3n, n)$-bisections.
\end{corollary}

\section{a 4-section obtained from a bisection}\label{sec:4-sec}
In this section, we consider the 4-section obtained by doubling a bisection as a multisection constructed immediately from a bisection. 
Let $X=X_1\cup X_2$ be a bisection of $X$. Then, we can easily construct a 4-section of the double of $X$ by identifying two copies of the bisection of $X$ along their Heegaard splittings of the boundaries.
Let $D(\mathcal{B})$ be the 4-section for a bisection $\mathcal{B}$. 


It is naturally raised that whether 4-sections obtained by doubling non-isotopic bisections are non-isotopic or not.
In the following, we consider this question for the bisection constructed in Section 4. Recall that this is a bisection of $M^\circ \times I$ for a 3-manifold $M$.

\begin{lemma}\label{lem:spine3}
	Let $B(\Sigma)=X_1\cup X_2$ be a bisection of a 4-manifold $X=M^{\circ} \times I$ obtained from a Heegaard splitting $V_1\cup_{\Sigma} V_2$, $S_1$ a spine of $X_1$, and $\langle s_1, \dots, s_g\mid R\rangle$ a presentation of $\pi_1(X)$, where $s_i$ is a homotopy class of  loops in $S_1$.
	Then $S_1$ induces the same presentation for $\pi_1(DX)$.
\end{lemma}

\begin{proof}
	By the construction of $B(\Sigma)$, we can assume that the boundary of $X$ is a connected sum of $(V_1\cup V_2)\times \{0 \}$ and $(V_1\cup V_2)\times \{ 1 \}$.
	Suppose that the representation of $\pi_1(M)$ induced by $(V_1\cup V_2)\times \{ 0 \}$ is 
	\[
		\langle v_1, \ldots, v_g \mid r\rangle,
	\]
	where $r$ is a set of relations induced by the 2-handles. Then $(V_1\cup V_2)\times \{ 1 \}$ has the same presentation.
	To distinguish them, we write it as
	\[
		\langle v_1', \ldots, v_g' \mid r'\rangle.
	\] 
	Then, $\pi_1(\partial X)=\pi_1(M\# -M)$ has the following presentation by Van Kampen's theorem:
	\[
		\langle v_1, \ldots, v_g, v_1', \ldots, v_g'  \mid r, r'\rangle.
	\]
	We show that after applying Van Kampen's theorem to $\pi_1(X)$ and $\pi_1(M\# -M)$, the presentation of $\pi_1(DX)$ will be the same as $\langle s_1, \dots, s_g\mid R\rangle$.
	
	We can assume that spines of both $V_1\times \{0\}$ and $V_1\times \{1\}$ are isotopic to $S_1$ in $X$.
	Let $\iota_\ast : \pi_1(M\# -M) \rightarrow \pi_1(X)$ be the homomorphism induced by the inclusion $\iota: M\# -M\hookrightarrow X$.  
	Then, the image of $v_1, \ldots, v_g, v_1', \ldots, v_g'$ under this homomorphism is
	\[
		\iota_\ast (v_i)=s_i , \iota_\ast (v_i')=s_i \mbox{ for $i=1, \ldots , g$}
	\]
	since the spines of both $V_1\times \{0\}$ and $V_1\times \{1\}$ are isotopic to $S_1$.
	By Van Kampen's theorem, we obtain the following presentation of $\pi_1(DX)$
	\[
		\langle s_1, \dots, s_g \mid R\rangle
	\]
	since $R$ is the same as $r$ and $r'$ by Lemma \ref{lem:spine}.
\end{proof}

\begin{lemma}
	Let $\mathcal{B}=X_1\cup X_2$ and $\mathcal{B}'=X_1'\cup X_2'$ be bisections with $\mathcal{N}(X_1)\neq \mathcal{N}(X_1')$ constructed in Section 4.
	Then, the 4-sections $D(\mathcal{B})$ and $D(\mathcal{B}')$ are not isotopic to each other.
\end{lemma}

\begin{proof}
	By Lemma \ref{lem:spine3}, we may identify $\pi_1(X)$ and $\pi_1(DX)$ so that spines of the bisections of $X$ for $X_1$ and $X_1'$ can be identified with them of the 4-sections.
         Since $\mathcal{N}(X_1)\neq \mathcal{N}(X_1')$ in $\pi_1(X)$,  we also have $\mathcal{N}(X_1)\neq \mathcal{N}(X_1')$ in $\pi_1(DX)$. By Lemma \ref{lem:fund2}, $D(\mathcal{B})$ and $D(\mathcal{B}')$ are mutually non-isotopic.
\end{proof}

Consequently, we obtain the following:

\begin{theorem}\label{thm:4-sec}
	For every $n\geq 2$, there exist 4-manifolds which admit non-isotopic $(2n, n)$ 4-sections of minimal genus.
\end{theorem}

\begin{proof}
We show that the 4-section genus of the 4-manifolds in Theorem \ref{thm:4-sec} is $2n$. Suppose that the 4-section genus is not $2n$. Then, we can destabilize the 4-section more than once. If we destabilize the 4-section once, then the two tuple $(g,k_1)$ of the type $(g;k_1,k_2,k_3,k_4)$ of the destabilized 4-section is $(2n-1,n)$ or $(2n-1,n-1)$. If $(g,k_1)=(2n-1,n)$, then the genus of the trisection obtained by Proposition 8.4 in \cite{IN} is $3n-2$. If $(g,k_1)=(2n-1,n-1)$, then the genus of the trisection obtained by the same method is $3n-1$. However, both cases contradict Corollary 7.2 in \cite{I} (see also Remark \ref{rem:spin}). One can lead the contradiction in the same way if destabilizations are performed more than twice.
\end{proof}

\begin{remark}\label{rem:spin}
The double of $X=M^{\circ} \times I$ is diffeomorphic to the spin of $M$. Therefore, Theorem \ref{thm:4-sec} is a 4-sectional analogy of \cite[Corollary 7.2]{I}.
\end{remark}

Finally, we depict the 4-section diagram with respect to the above 4-section. Since the double of a bisected 4-manifold is obtained by gluing two copies of the 4-manifold identically, the new 1-handlebody $H_4$ plays the same role as $H_2$ in the central surface. Therefore, for a bisection diagram, if we depict the parallel copy of the blue curves as orange curves, the resulting diagram is the desired 4-section diagram. For example, Figure \ref{4-secdiag} is the 4-section diagram when $M=L(2,1)$. One can easily check that this 4-section diagram describes the double of $L(2,1)^{\circ} \times I$, that is, the spun $L(2,1)$, or equivalently, Pao's manifold $L_2$ via Kirby calculus.

\begin{figure}[h]
		\centering
		\includegraphics[scale=1]{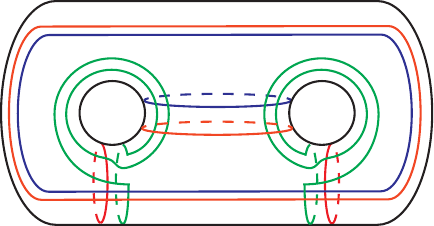}
		\caption{This figure is the 4-section diagram concerning the 4-section obtained by doubling the bisection of $L(2,1)^\circ \times I$ constructed in Proposition \ref{prop:construction}. Compare with Figure \ref{bidiag1}.}
		\label{4-secdiag}
	\end{figure}

\section{A multisection obtained from a modified bisection}\label{sec:multi0}
In this section, as a generalization of Theorem \ref{thm:4-sec}, we construct an $n$-section of a 4-manifold with boundary and infinitely many 4-manifolds with non-isotopic $n$-sections for arbitrary $n$ by modifying a bisection.
Let $X=X_1\cup X_2$ be a bisection such that $\partial X\cap X_1=H_1$, $X_1\cap X_2=H_2$ and $\partial X\cap X_2=H_3$.

We regard $H_2\cup H_2$ as the double of $H_2$. 
Then, $H_2\cup H_2$ is a genus $g$ Heegaard splitting of $\#^g S^1\times S^2$.
On the other hand, we may consider $\natural{^k }S^1\times B^3$ and its boundary $\partial (\natural{ ^g} S^1\times B^3)\cong \#^g S^1\times S^2$.
The boundary $\partial (\natural{ ^g} S^1\times B^3)$ admits a Heegaard splitting as follows; we decompose $\partial B^3$ into two disks $D_1$ and $D_2$ that are identified by their boundaries. Then $\partial (\natural{ ^g} S^1\times B^3)$ is decomposed into $(\natural{ ^g} S^1\times D_1)\cup (\natural{ ^g} S^1\times D_2)$, and thus we can identify this Heegaard splitting with $H_2\cup H_2$. 

Let $X'=X_1\cup X_2\cup X_3$ be a 4-manifold satisfying the following conditions:

\begin{itemize}
        \item $X_3$ is a 1-handlebody,
	\item $X_1\cap X_3=H_2$, $X_2\cap X_3=H_2$, and
	\item $\partial X'=H_1\cup H_3$.
\end{itemize}

Then $X_1\cup X_2\cup X_3$ is a trisection with boundary $H_1\cup H_3$.

\begin{lemma}
	The 4-manifold $X'$ is diffeomorphic to $X$.
\end{lemma}
\begin{proof}
	The multisection diagram of $X_1\cup X_2\cup X_3$ is obtained by adding to a bisection diagram of the bisection $X_1 \cup X_2$ a family of simple closed curves corresponding to the boundaries of the complete meridian disk system of the other $H_2$.
	By the construction of a Kirby diagram from a multisection diagram in \cite{IN}, the obtained Kirby diagram is the same as the Kirby diagram obtained from the bisection diagram of $X_1\cup X_2$.
\end{proof}

We will decompose $\natural^g S^1\times B^3$ into $n$ copies of $\natural^g S^1\times B^3$.
First of all, we consider a decomposition of $B^3$.
Let $B_1, B_2, \cdots , B_n$ be 3-balls and $B_i\cap B_{i+1}=D_i$ a disk for $i=1, \ldots , n-1$,  otherwise $B_{i}\cap B_{j}$ is empty. Then, $B_1\cup B_2\cup \cdots \cup B_n$ is a 3-ball.
Then we obtain the decomposition 
\[
	\natural^g S^1\times B^3 = (\natural^g S^1\times B_1)\cup (\natural^g S^1\times B_2)\cup \cdots \cup (\natural^g S^1\times B_n).
\]
This decomposition satisfies
\begin{itemize}
	\item $X_i=(\natural^g S^1\times B_i)\cong \natural^g S^1\times B^3$ for any $i$,
	\item $H_{ij}=X_i\cap X_j\cong \natural^g (S^1\times B^2)$ if $|i - j|=1$ otherwise $X_i\cap X_j\cong \Sigma_g$,
	\item $\partial X_1=H_1\cup H_{12}$ and $\partial X_n=H_n\cup H_{n(n-1)}$ are genus $g$ Heegaard splittings, and
	\item $\partial (\natural^g S^1\times B^3) = H_1\cup H_n$ is a genus $g$ Heegaard splitting of $\partial (\natural^gS^1\times B^3)$.
\end{itemize}
A schematic picture of this decomposition is shown in Figure \ref{multisection_1}.
\begin{figure}[h]
		\centering
		\includegraphics[scale=0.7]{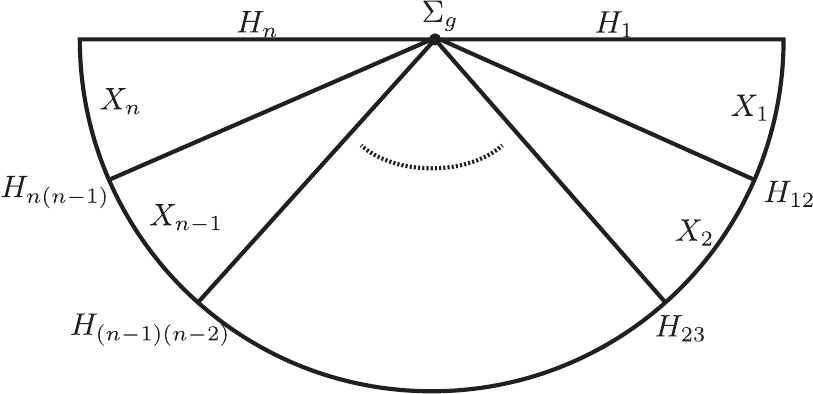}
		\caption{This is a schematic picture of the $n$-section of the 4-manifold with boundary.}
		\label{multisection_1}
	\end{figure}
By considering the above decomposition of $\natural^g S^1\times B^3$ instead of $X_3$ as above, we can obtain an $n$-section of $X$.

After applying this construction to a bisection we considered in Section \ref{sec:bisect_from_HS}, we can obtain an analogue of Lemma \ref{lem:spine}.

\begin{lemma}\label{lem:spine3}
Suppose that  $X=X_1\cup X_2$ is a bisection obtained in Section 4.
	Let $X=X_1\cup X_2\cup \cdots \cup X_n$ be an n-section of the 4-manifold $X$ obtained from the bisection $X_1 \cup X_2$ and $S_1$ a spine of $X_1$.
	Then $S_1$ induces the same presentation for $\pi_1(X)$ as appears in Lemma \ref{lem:spine}.
\end{lemma}

\begin{proof}
	Let $X_1\cup X_2\cup \cdots \cup X_n$ be an $n$-section obtained from a bisection $X_1\cup X_2$.
	Let $H_1=\partial X_1\cap \partial X$ and $H_n=\partial X_n\cap \partial X$.
	By the construction of the bisection $X_1\cup X_2$, $H_1$ gives a generating set of the fundamental group $\pi_1(X)$ of the $n$-sected 4-manifold $X$ and the spine of $X_1$ gives the same generating set. 
	Since a complete meridian disk system of each $H_{ij}$ is the same if $|i - j|=1$, we see that one of the presentations of $\pi_1(X)$ using this generating set is the same as in Lemma \ref{lem:spine}.
\end{proof}

By gluing the above $n$-section and the bisection, we have the following theorem from Lemma \ref{lem:fund2} and Theorem \ref{thm:bisec};

\begin{theorem}\label{thm:multi}
	For every $n\geq 2$ and $m\geq 4$, there exist 4-manifolds that admit non-isotopic genus $2n$ m-sections.
\end{theorem}

\begin{remark}
In the above construction of the $n$-section, some of the sectors satisfy $k_i\neq k_j$, where $k_i$ is the number of the boundary connected sum of $S^1\times B^3$ for each $i$. Namely, the $n$-section is unbalanced.
\end{remark}
 
 
 \begin{remark}\label{rem:multi}
 The above construction does not change the diffeomorphism type of the 4-manifold, that is, the resulting 4-manifold is $D(M^{\circ} \times I) \cong Spin(M)$. Therefore, Theorem \ref{thm:multi} is a multisectional analogy of \cite[Corollary 7.2]{I}. Note it is obvious that Theorem \ref{thm:4-sec} corresponds to the case $m=4$ in Theorem \ref{thm:multi}.
 \end{remark}
 
\section{A multisection obtained from several bisections}\label{sec:multi}

In this section, we consider a multisection of a 4-manifold obtained by gluing several bisections constructed in Proposition \ref{prop:construction} as another generalization of Theorem \ref{thm:4-sec}. Let $X \cup_H Y$ denote the 4-manifold obtained by gluing 4-dimensional 1-handlebodies $X$ and $Y$ identically along a 3-dimensional 1-handlebody $H$.

\begin{lemma}\label{lem:H_1}
For two copies $X$ and $X^{'}$ of the bisection constructed in Proposition \ref{prop:construction}, $\pi_1(X \cup_{H_1} X^{'}) \cong \pi_1(X) \cong \pi_1(X^{'})$.
\end{lemma}

\begin{proof}
Let $\langle S_1  \mid R_1\rangle$ and $\langle S_{1}^{'}  \mid R_{1}^{'}\rangle$ be presentations of $\pi_1(X)$ and $\pi_1(X^{'})$ respectively, where $S_1$ and $S_{1}^{'}$ are generating sets coming from spines of $X_1$ and $X_{1}^{'}$ respectively. Then, by Van Kampen's theorem, we have 
\[
\pi_1(X \cup_{H_1} X^{'}) \cong \langle S_1, S_{1}^{'} \mid R_1, R_{1}^{'}, I_1(y)=I_2(y)\ (y \in Y)\rangle,
\]
where $\pi_{1}(H_1)=\langle Y \mid \emptyset \rangle$ and $I_1 \colon \pi_1(H_1) \to \pi_1(X)$ and $I_2 \colon \pi_1(H_1) \to \pi_1(X^{'})$ are homomorphisms induced from the inclusions $i_1 \colon H_1 \to X$ and $i_2 \colon H_1 \to X^{'}$ respectively. Let $Y=\{v_1, \ldots, v_g, w_1, \ldots, w_g\}$, where $v_1, \ldots, v_g$ are coming from a spine of $(V_1 \setminus N(p)) \times 0 \subset H_1$ and $w_1, \ldots, w_g$ are coming from a spine of $(V_1 \setminus N(p)) \times 1 \subset H_1$. Then, we have $I_1(v_i)=I_2(w_i)=s_i$ and $I_1(v_{i})=I_2(w_i)=s_{i}^{'}$ for all $i \in \{1, \ldots, g\}$ since these spines are isotopic to a spine of $X_1$, where $s_i$ and $s_{i}^{'}$ are i-th generating elements of $S_1$ and $S_{1}^{'}$ respectively. Thus, we have 
\[
\pi_1(X \cup_{H_1} X^{'}) \cong \langle S_1, S_{1}^{'} \mid R_1, R_{1}^{'}, s_i=s_{i}^{'}\ (i \in \{1, \ldots, g\})\rangle \cong \langle S_1  \mid R_1\rangle.
\]
\end{proof}

\begin{lemma}\label{lem:H_3}
For two copies $X$ and $X^{'}$ of the bisection constructed in Proposition \ref{prop:construction}, $\pi_1(X \cup_{H_3} X^{'}) \cong \pi_1(X) \cong \pi_1(X^{'})$.
\end{lemma}

\begin{proof}
We can consider another presentation $\langle S_2  \mid R_2\rangle$ of $\pi_1(X)$, where $S_2$ is a generating set coming from a spine of $X_2$. Thus, we can use the argument in the proof of Lemma \ref{lem:H_1} in the same way.
\end{proof}

Using Lemma \ref{lem:H_1} and Lemma \ref{lem:H_3} repeatedly, we have the following lemma.

\begin{lemma}\label{lem:pi_1}
Let $X_1, \ldots, X_m$ be $m$ copies of the bisection. Then, $\pi_1(X_1 \cup_{H_1} X_2 \cup_{H_3} \dots \cup_{H_3} X_m) \cong \pi_1(X_i)$ if $m$ is odd and $\pi_1(X_1 \cup_{H_1} X_2 \cup_{H_3} \dots \cup_{H_1} X_m) \cong \pi_1(X_i)$ if $m$ is even.
\end{lemma}

Let $X^{m}$ denote $X_1 \cup_{H_1} X_2 \cup_{H_3} \dots \cup_{H_3} X_m$ if $m$ is odd and $X_1 \cup_{H_1} X_2 \cup_{H_3} \dots \cup_{H_1} X_m$ if $m$ is even (see Figure \ref{fig:Xm}).

\begin{figure}[h]
\begin{tabular}{cc}
\begin{minipage}{0.45\hsize}
\begin{center}
\includegraphics[width=6cm, height=3cm, keepaspectratio, scale=1]{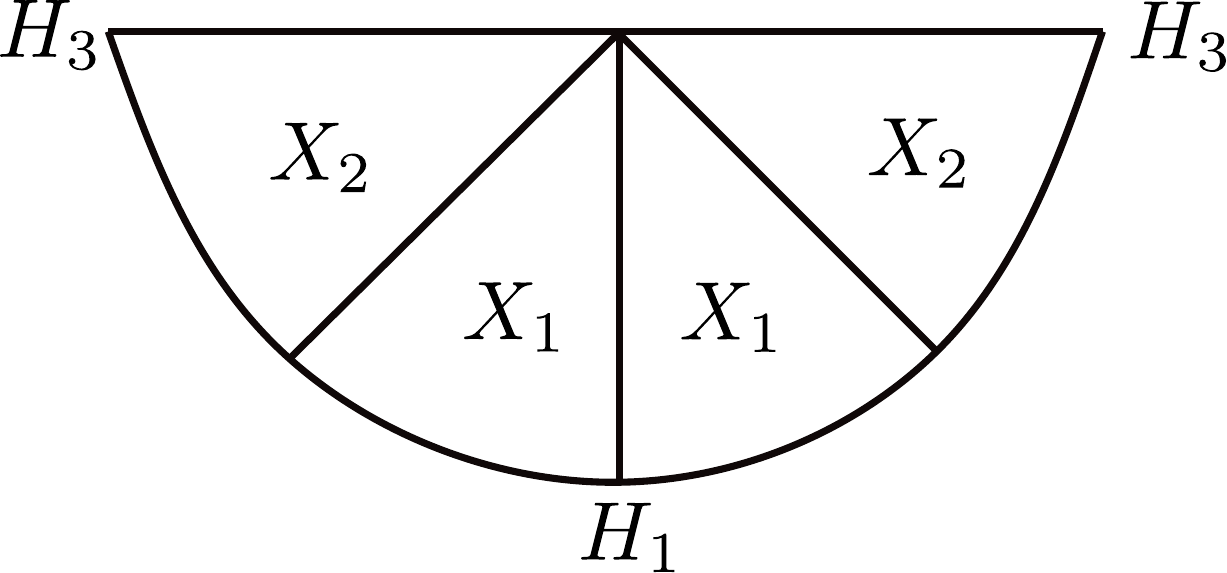}
\end{center}
\setlength{\captionmargin}{50pt}
\subcaption{$m=2$.}
\label{fig:}
\end{minipage} 
\begin{minipage}{0.55\hsize}
\begin{center}
\includegraphics[width=6cm, height=3cm, keepaspectratio, scale=1]{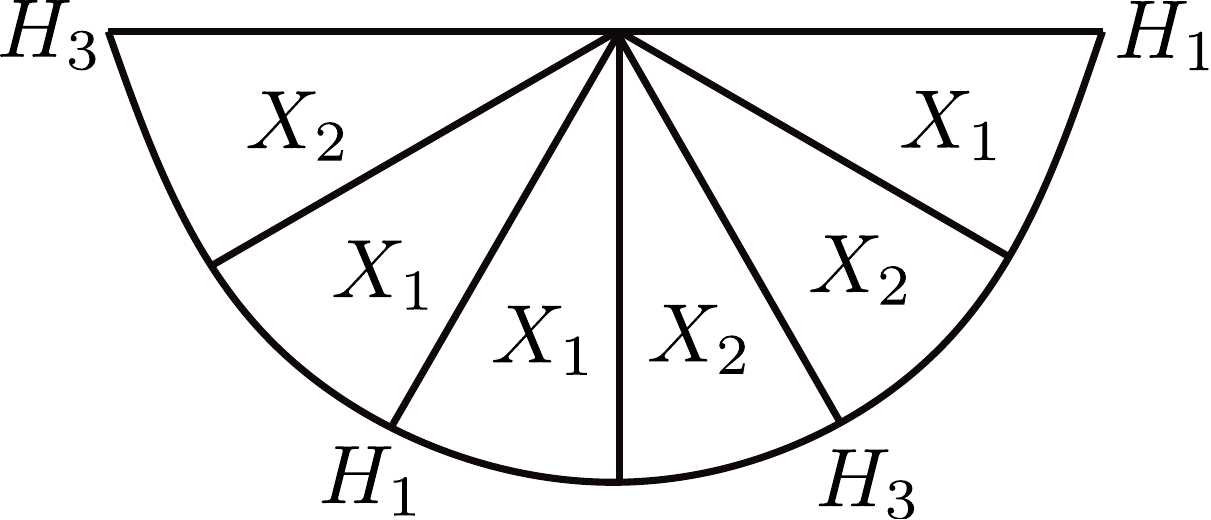}
\end{center}
\setlength{\captionmargin}{50pt}
\subcaption{$m=3$.}
\label{fig:}
\end{minipage} 
\end{tabular}
\setlength{\captionmargin}{50pt}
\caption{A schematic picture of the 4-manifold $X^{m}$.}
\label{fig:Xm}
\end{figure}

\begin{lemma}\label{lem:boundary}
The boundary $\partial{X^{m}}$ is $\partial{X_i}$ if $m$ is odd and $\#_{2g} S^1 \times S^2$ if $m$ is even.
\end{lemma}

\begin{proof}
Let $\partial{X_i} = H_1(i) \cup_{\phi_i} H_3(i)$, where $\phi_i \colon \partial{H_3(i)} \to \partial{H_1(i)}$. Then, $\partial{X^{2}} = H_3(1) \cup_{\psi} H_3(2)$, where $\psi = \phi_{2}^{-1} \circ id_{H_1(1)} \circ \phi_{1} \colon \partial{H_3(1)} \to \partial{H_1(1)} = \partial{H_1(2)} \to \partial{H_3(2)}$. Since $\phi_1 = \phi_2$, we have $\psi = id$. Thus, we have $\partial{X^{2}} = \#_{2g} S^1 \times S^2$. Similarly, $\partial{X^{3}} = H_3(1) \cup_{\omega} H_1(3)$, where $\omega = \phi_3 \circ \psi = \phi_3$. Thus, we have $\partial{X^{3}} = \partial{X_i}$. We can show the statement by repeating this argument.
\end{proof}

If $m$ is odd, let $X_0$ denote a copy of the bisection. If $m$ is even, let $X_0$ denote the bisection constructed in Proposition \ref{prop:construction} in the case $M = \#_{g} S^1 \times S^2$ ($\partial{X_0} = M \# -M = \#_{2g} S^1 \times S^2$). Since $\partial{X^{m}} = \partial{X_0}$ (Lemma \ref{lem:boundary}) for any integer $m \ge 1$, we can obtain a $4\ell+2$- (resp. $4\ell$-) section of $X^{m} \cup_{\partial} X_0$ if $m$ is even (resp. odd).

\begin{lemma}
For any integer $m \ge 1$, $\pi_1(X^{m} \cup_{\partial} X_0) \cong \pi_1(X_i)$.
\end{lemma}

\begin{proof}
By Lemma \ref{lem:pi_1}, one can show the statement by the argument in the proof of Lemma \ref{lem:H_1}.
\end{proof}

Therefore, we can construct non-isotopic multisections of 4-manifolds with even sectors in the same way as in Theorem \ref{thm:4-sec}. The minimality of genera is obtained by the consequence of Theorem \ref{cor_1}.

\begin{theorem}\label{thm:even}
For every $n \ge 2$ and $m \ge 2$, there exist 4-manifolds that admit non-isotopic $(2n,n)$ $2m$-sections of the minimal genus.
\end{theorem}


In the construction of $X^{m}$, $X_1 \cup_{H_1} X_1$ is diffeomorphic to $\natural_{g} S^1 \times B^3$. This is because $\partial{X_1} = H_1 \cup H_2$ is a Heegaard splitting of $\#_{2g} S^1 \times S^2$. Since $H_2 \cup H_2$ is the double of $H_2$ by the construction of $X^{m}$, $2g$ simple closed curves corresponding to the boundaries of the complete meridian disk system of $H_2$ are parallel to them of $H_2$ in the Heegaard diagram of the Heegaard splitting $\partial{X_1} = H_1 \cup H_2$. Therefore, the resulting bisection diagram describes $\natural_{g} S^1 \times B^3$.  

Thus, we can regard $X_1 \cup_{H_1} X_1$ as one sector. Namely, we can think of the number of sectors of $X^{m}$ as $4\ell+1$ if $m$ is odd and $4\ell+3$ if $m$ is even for some $\ell$. Therefore, we can construct non-isotopic multisections of 4-manifolds with odd sectors in the same way as in Theorem \ref{thm:4-sec}. The minimality of genera is obtained by the consequence of Theorem \ref{cor_1}.

\begin{theorem}\label{thm:odd}
For every $n \ge 2$ and $m \ge 2$, there exist 4-manifolds that admit non-isotopic $(2n,n)$ $2m+1$-sections of the minimal genus.
\end{theorem}

Combining Theorem \ref{thm:even} and Theorem \ref{thm:odd}, we can obtain non-isotopic multisections of 4-manifolds with any number of sectors.

\begin{remark}
Unlike Remark \ref{rem:multi}, the above construction may change the diffeomorphism type of the 4-manifold.
\end{remark}

\section{Acknowledgement}
The first author would like to express sincere gratitude to his supervisor, Hisaaki Endo, for his support and encouragement throughout this work. The first author was partially supported by Grant-in-Aid for JSPS Research Fellow from JSPS KAKENHI Grant Number JP23KJ0888. The authors would like to thank the referee for advising the construction of non-isotopic multisections of 4-manifolds with any number of sectors.

\end{document}